\documentclass[12pt]{amsart}
\usepackage{wrapfig}
\usepackage{graphicx}
\usepackage[all]{xy}
\usepackage{amsmath,amsthm,amssymb}
\usepackage{mathabx,epsfig}

\usepackage[margin=1in]{geometry}
\usepackage{amsmath,amscd}
\usepackage{mathtools}

\usepackage[all]{xy}

\usepackage{thmtools}
\usepackage{thm-restate}
\usepackage{hyperref}
\usepackage{cleveref}

\newtheorem{theorem}{Theorem}[section]
\newtheorem{conjecture}[theorem]{Conjecture}
\newtheorem{corollary}[theorem]{Corollary}
\newtheorem{lemma}[theorem]{Lemma}
\usepackage[small,compact]{titlesec}
\usepackage{enumerate}
\usepackage{enumitem}
\usepackage{times}
\makeatletter

\newcommand{\Rmnum}[1]{\expandafter\@slowromancap\romannumeral #1@}
\makeatother

 \newcommand{\GL}{\operatorname{GL}}
  \newcommand{\Supp}{\operatorname{Supp}}

 \newcommand{\s}{\operatorname{S}}
 \newcommand{\Spec}{\operatorname{Spec}}
 \newcommand{\g}{\operatorname{\gamma}}
 \newcommand{\N}{\operatorname{N}}
 \newcommand{\Map}{\operatorname{}}
 \newcommand{\PrePer}{\operatorname{PrePer}}
 \newcommand{\Per}{\operatorname{Per}}
 
\newcommand*{\mysqrt}[4]{\sqrt[\leftroot{#1}\uproot{#2}#3]{#4}}
 \newcommand{\genus}{\operatorname{genus}}
 
 \newcommand{\Nat}{\operatorname{Nat}}
 \newcommand{\Dir}{\operatorname{Dir}}
\newcommand{\widesim}[2][1.5]{\mathrel{\overset{#2}{\scalebox{#1}[1]{$\sim$}}}}
 
\newcommand{\Mod}[1]{\ (\textup{mod}\ #1)}

  \newcommand{\Res}{\operatorname{Res}}
 
\newcommand{\Avg}{\operatorname{Avg}}

 \newcommand{\Aut}{\operatorname{Aut}}

\newcommand{\Gal}{\operatorname{Gal}}

\newcommand{\disc}{\operatorname{disc}}
\usepackage{amssymb,fge}
\newcommand{\mysetminus}{\mathbin{\fgebackslash}}

\usepackage{xcolor}
\usepackage{titlesec}
\titleformat{\section}[block]{\color{black}\large\filcenter}{}{1em}{}
\titleformat{\subsection}[hang]{\bfseries}{}{1em}{}
\theoremstyle{remark}
\newtheorem{definition}{Definition} 
\newtheorem{remark}[theorem]{Remark}
\begin{document}
\title[Zsigmondy sets, Galois groups, and the Kodaira-Spencer map]{Average Zsigmondy sets, dynamical Galois groups,\\ and the Kodaira-Spencer map}
\author{Wade Hindes}
\date{\today}
\maketitle
\renewcommand{\thefootnote}{}
\begin{abstract} Let $K$ be a global function field and let $\phi\in K[x]$. For all wandering basepoints $b\in K$, we show that there is a bound on the size of the elements of the dynamical Zsigmondy set $\mathcal{Z}(\phi,b)$ that depends only on $\phi$, the poles of the $b$, and $K$. Moreover, when we order $b\in\mathcal{O}_{K,S}$ by height, we show that $\mathcal{Z}(\phi,b)$ is empty on average. As an application, we prove that the inverse limit of the Galois groups of iterates of $\phi(x)=x^d+f$ is a finite index subgroup of an iterated wreath product of cyclic groups. Finally, we establish similar results on Zsigmondy sets when $K$ is the field of rational numbers or $K$ is a quadratic imaginary field subject to an added stipulation: either zero has finite orbit under iteration of $\phi$ or the Vojta conjecture for algebraic points on curves holds.          
\end{abstract}
\begin{section}{1. Introduction}
Given a rational map $\phi\in K(x)$ over a global field $K$ and a basepoint $b\in\mathbb{P}^1(K)$, we study the prime factors of $\phi^n(b)$ as we iterate $\phi$. Specifically, we are interested in knowing whether or not $\phi^n(b)$ has a primitive prime factor, that is, whether or not there is a prime dividing $\phi^n(b)$ that does not divide any lower order iterates. This problem is analogous to a classical problem of Bang \cite{Bang}, Zsigmondy \cite{Zsig}, and Schinzel \cite{Schinzel} on the prime factorization of integer sequences defined dynamically on the multiplicative group.
 
Our motivation for studying primitive prime divisors comes from the Galois theory of iterates. For instance, in the family $\phi_f(x)=x^d+f$, the existence of $d$-power free primitive prime divisors in the orbit of zero implies a dynamical version of Serre's open image theorem (\cite[Theorem 25]{xdc}, \cite[Theorem 3.3]{RJones} and \cite{Serre-reps}), and we prove this over global function fields.
 
To begin, we fix some notation. Let $K$ be a global field and let $V_K$ be a complete set of valuations on $K$ (corresponding to prime ideals). We say that $v\in V_K$ is a \emph{primitive prime divisor} of $\phi^n(b)$ if
\[v(\phi^n(b))>0\;\;\text{and}\;\; v(\phi^m(b))=0\; \text{for all}\; 1\leq m\leq n-1\]
such that $\phi^{m}(b)\neq0$. Likewise, we define the \emph{Zsigmondy set} of $\phi$ and $b$ to be
\[\mathcal{Z}(\phi,b):=\{n\;|\;\phi^n(b)\;\text{has no primitive prime divisors}\}.\]
 
Over number fields, there are many results regarding the finiteness and size of $\mathcal{Z}(\phi,b)$ in special families; see for example \cite{ABCimplies,Ing-Silv,ellipticdivis,Krieger,Silv-Vojta}. However, it remains difficult to bound $\mathcal{Z}(\phi,b)$ in general. Nevertheless, over function fields $K/\mathbb{F}_q(t)$, we show that there is a bound on $\#\mathcal{Z}(\phi,b)$ depending only on $\phi$, the poles of $b$, and on $K$; see Theorem \ref{thm:Avg}.
 
The main technique we use to prove this result is to associate to every element of $\mathcal{Z}(\phi,b)$ a rational points on some curve, and then use height bounds for points on curves to bound the size of the corresponding element of the Zsigmondy set. As one may expect, there are certain complications that arise in characteristic $p>0$ if the associated curve is defined over the field of constants. Therefore, we need the following geometric condition:
\begin{definition}{\label{def:isotriv}} Let $K/\mathbb{F}_q(t)$ and $\ell\geq2$ be an integer coprime to the characteristic of $K$. Then we say that $\phi$ is dynamically $\ell$-power non-isotrivial if there exists an integer $m\geq1$ such that
\begin{equation}{\label{dynamicalcurve}} C_{\ell,m}(\phi): Y^\ell=\phi^m(X)=(\underbracket{\phi\circ\phi\circ\dots\circ\phi}_m)(X)
\end{equation}
is a non-isotrivial curve (meaning that the associated Kodaira-Spencer map is non-zero on some open set \cite{htineq}) of genus at least $2$. As a motivating example, we show that $\phi_f(x)=x^d+f$ is dynamically $2$-power and $d$-power non-isotrivial over $K=\mathbb{F}_p(t)$; see Theorem \ref{thm:unicrit} below. 
\end{definition}
In particular, if $S\subseteq V_K$ is any finite subset, $\mathcal{O}_{K,S}$ is the ring of $S$-integers of $K$, and $\phi$ is dynamically $\ell$-power non-isotrivial for some $\ell\geq2$, then we show that
\begin{equation}
\mathcal{Z}(\phi,S):=\big\{n\,|\,n\in\mathcal{Z}(\phi,b)\;\text{for some}\;b\in\mathcal{O}_{K,S},\; \hat{h}_{\phi}(b)>0\big\}
\end{equation}
is finite; in particular, $\#\mathcal{Z}(\phi,b)$ is uniformly bounded over all $b\in\mathcal{O}_{K,S}$. Moreover, we prove an analogous result when $K=\mathbb{Q}$ or $K$ is an imaginary quadratic field, assuming the Vojta conjecture.
 
On the other hand, one expects that $\mathcal{Z}(\phi,b)$ is empty if we choose the basepoint $b\in K$ ``at random". A reasonable interpretation of this statements can be formulated in terms of averages. To wit, for all $B\geq0$, let $\mathcal{O}_{K,S}(B)$ be the set of points of $\mathcal{O}_{K,S}$ of height at most $B$, a finite set by \cite[Theorem 3.7]{Silv-Dyn}. Then we study the average
\begin{equation} \widebar{\Avg}(\mathcal{Z}(\phi),S):=\limsup_{B\rightarrow\infty}\frac{\sum_{\,b\in\mathcal{O}_{K,S}(B)}\,\#\mathcal{Z}(\phi,b)}{\#\mathcal{O}_{K,S}(B)}.
\end{equation}
as we vary over all $b\in\mathcal{O}_{K,S}$. In particular, we show that $\widebar{\Avg}(\mathcal{Z}(\phi),S)=0$ for all $S$; in other words, the naive heuristic is correct: one expects to see primitive prime divisors at every stage of iteration. Both the uniform bound and average-result are summarized below (in what follows, $\Per(\phi)$ and $\PrePer(\phi)$ denote the set of periodic and preperiodic points of $\phi$ respectively): 
\begin{restatable}{thm}{Average}{\label{thm:Avg}} Let $\phi\in K[x]$ be such that $\deg(\phi)\geq2$ and $0\not\in\Per(\phi)$.
\begin{enumerate}[topsep=2.2mm,itemsep=2.5mm]
\item[\textup{(1)}] When $K=\mathbb{Q}$ or $K$ is a quadratic imaginary field, we have the following cases:
\begin{enumerate}[topsep=2.5mm,itemsep=2.5mm]
\item[\textup{(a)}] If $0\in\PrePer(\phi)$, then $\mathcal{Z}(\phi,S)$ is finite and $\widebar{\Avg}(\mathcal{Z}(\phi),S)=0$.
\item[\textup{(b)}] If $0\not\in\PrePer(\phi)$ and the Vojta conjecture \cite[Conj. 25.1]{Vojta} holds, then $\mathcal{Z}(\phi,S)$ is finite \\and $\widebar{\Avg}(\mathcal{Z}(\phi),S)=0$.  
\end{enumerate} 
\item[\textup{(2)}] If $K/\mathbb{F}_q(t)$ and $\phi$ is dynamically $\ell$-power non-isotrivial for some $\ell\geq2$, then $\mathcal{Z}(\phi,S)$\\ is finite and $\widebar{\Avg}(\mathcal{Z}(\phi),S)=0$.     
\end{enumerate}
In other words, there is a bound on the elements of $\mathcal{Z}(\phi,b)$ depending only on $\phi$, the poles of $b$, and $K$. Moreover, if we order $\mathcal{O}_{K,S}$ by height, then $\mathcal{Z}(\phi,b)$ is empty on average.    
\end{restatable}
As an application, we use Theorem \ref{thm:Avg} to study dynamical Galois groups. For $n\geq1$, let $K_n(\phi)$ be the field obtained by adjoining all solutions of $\phi^n(x)=0$ to $K$. Generically, the extension $K_n(\phi)/K$ is Galois, and we let $G_{K,n}(\phi):=\Gal_K(\phi^n)$ be the Galois group of $K_n(\phi)/K$. Since $K_{n-1}(\phi)\subseteq K_{n}(\phi)$ for all $n\geq1$ (under some mild separability assumptions), we may define  
\begin{equation}{\label{Galois}} G_K(\phi)=\lim_{\longleftarrow}G_{K,n}(\phi)
\end{equation}
with respect to the restriction maps. Dynamical analogs on $\mathbb{P}^1$ of the Galois representations attached to abelian varieties \cite{Serre-reps} (where one instead considers iterated preimages of multiplication maps), the groups $G_K(\phi)$ have obtained much attention in recent years; for instance, see \cite{BJ,GaloisUniform, RJSurvey, Odoni,Pink,Stoll-Galois}, among other places. 
 
Of course, a key difference in this setting is the lack of group structure on projective space, and as such $G_K(\phi)$ may often only be viewed as a subgroup of a wreath product (or the automorphism group of a tree) and not inside a group of matrices. Explicitly, let $T(d)$ denote the infinite $d$-ary rooted tree. If we write $\phi^n=f_n/g_n$ for some $f_n,g_n\in K[x]$ such that $\disc({f_n})\neq0$ for all $n\geq1$, then we may identify $T(d)$ with the set of iterated preimages of zero (under $\phi$) in $\widebar{K}$ with the edge relation given by evaluation; see \cite{BJ} for more details. In particular, $G_K(\phi)\leq\Aut(T(d))$, since Galois commutes with evaluation. 
 
As in the case of abelian varieties, one expects that $G_K(\phi)$ is a large subgroup of $\Aut(T(d))$. However, given the unruly nature of $\Aut(T(d))$ and its subgroups, we focus our attention on polynomials of a special form. To do this, fix a faithful permutation representation of the cyclic group $C_d\leq \s_d$ and let $W(d)$ be the infinite iterated wreath product  of $C_d$ acting on $T(d)$; see \cite{wreath} or \cite[Defs. 2.3 and 2.4]{Juul}. In particular, if $\mu_d$ is the group of $d$-th roots of unity in $\widebar{K}$ and $\phi_f(x)=x^d+f$ for some $f\in K$, then we have the refinement $G_{K(\mu_d)}(\phi_f)\leq W(d)\leq\Aut(T(d))$; see \cite[Lemma 2.5]{Juul}. Moreover, it follows from Theorem \ref{thm:Avg} and some calculations involving the Kodaira-Spencer map of $C_{m,\ell}(\phi)$, that $G_{K(\mu_d)}(\phi_f)\leq W(d)$ is a finite index subgroup:       
 
\begin{restatable}{thm}{unicrit}{\label{thm:unicrit}} Let $K=\mathbb{F}_p(t)$, let $f\in K$ and let let $d\geq2$ be coprime to $p$. If $f\not\in K^p$ and $\phi(x)=x^d+f$, then the following statements hold:
\begin{enumerate}[topsep=2.25mm,itemsep=2.5mm]
\item[\textup{(1)}] $\mathcal{Z}(\phi,S)$ is finite and $\widebar{\Avg}(\mathcal{Z}(\phi),S)=0$\, for all finite subsets $S\subseteq V_K$.
\item[\textup{(2)}] If $d$ is an odd prime, $d\not\equiv1\Mod{p}$, $f\in\mathcal{O}_K$ and $f\not\in {K(\mu_d)}^d$, then $G_{K(\mu_d)}(\phi)\leq W(d)$ is a finite index subgroup.
\end{enumerate}    
\end{restatable}
In fact, it follows from Theorem \ref{thm:unicrit} that $G_{K(\mu_d)}(\phi_f)\leq W(d)$ is a finite index subgroup for all non-constant $f\in\mathbb{F}_p(t)$ and $d\not\equiv1\Mod{p}$; see Remark \ref{rem:eventuallystable} below. In particular, the number of irreducible factors of $\phi^n$ over $\mathbb{F}_p(t)$ is bounded independently of $n$ (a phenomenon called eventual stability \cite[\S5]{RJSurvey}), and we recover a stronger version of \cite[Corollary 7]{xdc}. For applications of eventual stability to the study of integral points in reverse orbits, see \cite[\S3]{Alon-Rafe} and \cite[Theorem 2.6]{Sookdeo}. As for characteristic zero function fields $K/k(t)$, a finite index statement for $G_{K(\mu_d)}(x^d+f)\leq W(d)$ follows from \cite{ABCimplies} and \cite[Theorem 25]{xdc}. However, we can improve upon this result, making the index bounds explicit and uniform when $K=k(t)$ is a rational function field.    
\begin{restatable}{thm}{uniform}{\label{thm:uniform}}
Let $K=k(t)$ be a rational function field of characteristic zero, let $d$ be an odd prime, and let $\phi(x)=x^d+f$ for some non-constant $f\in k[t]$. Then the following statements hold:
\begin{enumerate}[topsep=2.3mm,itemsep=3mm]
\item[\textup{(1)}] If $f\notin K(\mu_d)^d$, then we have the index bound
\[ \log_d\, [W(d): G_{K(\mu_d)}]\leq\frac{d^{10}-1}{d-1}+10.\]  
\item[\textup{(2)}] If $\Gal_{K(\mu_d)}(\phi^{10})\cong [C_d]^{10}$, then $G_{K(\mu_d)}(\phi)\cong W(d)$.
\item[\textup{(3)}] If $d\geq367$ and $\Gal_{K(\mu_d)}(\phi^{5})\cong [C_d]^{5}$, then $G_{K(\mu_d)}(\phi)\cong W(d)$
\end{enumerate}
Moreover, if $\phi(x)=x^d+t$, then $G_{K(\mu_d)}(\phi)\cong W(d)$ for all $d\geq2$ (not necessarily prime). In particular, if $k$ is a number field and $\phi_c(x)=x^d+c$\, for some $c\in k$, then there are infinitely many values of $c$ satisfying $\Gal_{k(\mu_d)}(\phi_c^n)\cong [C_d]^n$.     
\end{restatable}
\indent \textbf{Acknowledgements:} It is a pleasure to thank Rafe Jones, Brian Conrad, Bjorn Poonen, and Felipe Voloch for the useful discussions related to the work in this paper.
\end{section}
\begin{section}{2. Primitive prime divisors and averages}
In this section, we make the following conventions:
\begin{itemize}
\item $K$ is a number field or a finite extension $K/\mathbb{F}_q(t)$.
\item if $K$ is a function field, $k$ is its field of constants of $K$.
\item $\mathfrak{p}$ is a finite prime of $K$.
\item $k_\mathfrak{p}$ is the residue field of $\mathfrak{p}$.
\item if $K$ is a number field, then $\N_\mathfrak{p}:=\frac{\log\#k_{\mathfrak{p}}}{[K:\mathbb{Q}]}$. 
\item if $K$ is a function field, then $\N_\mathfrak{p}:=[k_\mathfrak{p}:k]$.
\end{itemize} 
We normalize $\N_\mathfrak{p}$ in the number field case, since it streamlines our proofs. Moreover, given a finite set of primes $S\subseteq V_K$, we let $\mathcal{O}_K:=\{a\in K : v(a)\geq0,\; v\in V_K\}$ be the ring of integers of $K$ and $\mathcal{O}_{K,S}:=\{a\in K : v(a)\geq0,\; v\notin S\}$ be the ring of $S$-integers. Similarly, when $K$ is a function field, we fix a prime $\mathfrak{p}_0$ and set $\mathcal{O}_K:=\big\{\alpha\in K\, : \, v_{\mathfrak{p}}(\alpha)\geq0,\;\mathfrak{p}\neq\mathfrak{p}_0\big\}$ and let $\mathcal{O}_{K,S}:=\big\{\alpha\in K\, : \, v_{\mathfrak{p}}(\alpha)\geq0,\; \mathfrak{p}\notin S\big\}$ for all $S$ containing $\mathfrak{p}_0$. Moreover, we let $\mathcal{O}_K^*$ and $\mathcal{O}_{S,K}^*$ be the corresponding unit groups.  
 
We now define the relevant global Weil-height functions; see \cite[\S3.1]{Silv-Dyn} and \cite[Theorem 1.4.11]{ffields} for more details. If $K$ is a function field, then we define the height of $\alpha\in K$  to be 
\begin{equation}{\label{ffhtdef}} h(\alpha)=-\sum_{\mathfrak{p}\in V_K}\min(v_{\mathfrak{p}}(\alpha),0)\cdot\N_{\mathfrak{p}}=\sum_{\mathfrak{p}\in V_K}\max(v_{\mathfrak{p}}(\alpha),0)\cdot\N_{\mathfrak{p}}
\end{equation}
On the other hand, if $K$ is a number field, the height of $\alpha\in K$ is
\begin{equation}{\label{nfhtdef}}
h(\alpha)=-\sum_{\mathfrak{p}\in V_K}\min(v_{\mathfrak{p}}(\alpha),0)\cdot\N_{\mathfrak{p}}\;+\; \frac{1}{[K:\mathbb{Q}]}\sum_{\sigma:K\rightarrow\mathbb{C}}\max(\log|\sigma(\alpha)|,0).
\end{equation}
\begin{remark} The key advantage in the function field setting when studying primitive prime divisors is that one can compute heights of integers by keeping track of positive valuations only.   
\end{remark}
As was mentioned in the introduction, the main technique we use to prove Theorem \ref{thm:Avg} comes from the theory of rational points on curves. However, over function fields, we must stipulate that our curve not be defined over the field of constants, otherwise certain results (such as the Mordell conjecture) are false. The most convenient way to achieve this is to define the Kodaira-Spencer map.
 
To do this, we think of a curve $X_{/K}$ as a surface over $k$. In particular, $X$ is equipped with a map $f:X\rightarrow C$ to a curve $C_{/k}$ satisfying $K=k(C)$. The Kodaira-Spencer map (or $KS$) is constructed on any open set $U\subseteq C$ over which $f$ is smooth from the exact sequence 
\[0\rightarrow f^*\Omega_U^1\rightarrow\Omega_{X_{U}}^1\rightarrow \Omega_{X_{U}/U}^1\rightarrow 0.\]
by taking the coboundary map $KS: f_*(\Omega_{X_{U}/U})\rightarrow\Omega_U^1\times R^1f_*(\mathcal{O}_{X_{U}})$.       
\begin{proof}[(Proof Theorem \ref{thm:Avg})]
To estimate the size of elements in $\mathcal{Z}(\phi,b)$, we refine our proof of the finiteness of $\mathcal{Z}(\phi,b)$ in \cite[Theorem 1]{primdiv} and follow the conventions therein. Note that $S\subseteq S'$ implies $\mathcal{Z}(\phi,S)\subseteq\mathcal{Z}(\phi,S')$. Therefore, to prove that $\mathcal{Z}(\phi,S)$ is finite, we may enlarge $S$ and assume that
\[\text{(a)}.\;\;\;b\in\mathcal{O}_{K,S}\;\;\;\;\;\; \text{(b)}.\;\;\;\phi\in\mathcal{O}_{K,S}[x]\;\;\;\;\;\;\text{(c)}.\;\;\; v(a_d)=0\;\;\text{for all}\; v\notin S\;\;\;\;\;\; \text{(d)}.\;\;\; \mathcal{O}_{K,S}\;\text{is a UFD,}\] where $a_d$ is the leading term of $\phi$. Note that condition (d) is made possible by the finiteness of the class group; see \cite[Prop. 14.2 ]{Rosen} over function fields. Likewise, we may assume that $\mathfrak{p}_0\in S$. 
 
We first bound $\mathcal{Z}(\phi,b)$ when $0\not\in\mathcal{O}_\phi(b)$, true of all but finitely many $b\in K$ by Lemma \ref{0orbit}. To do this, we use the following decomposition of $\phi^n(b)$ into an $\ell$ and $\ell$-free part.
\begin{lemma}{\label{lemma:decomp}} Let $\phi$, $K$ and $S$ be as above and let $\ell\geq2$. Then we have a decomposition
\begin{equation}{\label{decomp}} \phi^n(b)=u_n\cdot d_n\cdot y_n^\ell,\,\;\;\;\text{for some}\;\; d_n,y_n\in\mathcal{O}_{K,S},\;u_n\in\mathcal{O}_{K,S}^*,
\end{equation}
satisfying the following properties:
\begin{enumerate}[topsep=2mm,itemsep=2mm]
\item[\textup{(1)}] $0\leq v_{\mathfrak{p}}(d_n)\leq\ell-1$ for all $\mathfrak{p}\not\in S$,
\item[\textup{(2)}] There is a constant $r(S)$ such that $0\leq v_{\mathfrak{p}}(d_n)\leq r(S)$ for all $\mathfrak{p}\in S$ when $K/\mathbb{Q}$ \\ and all $\mathfrak{p}\in S\mysetminus\{\mathfrak{p}_0\}$ when $K/\mathbb{F}_q(t)$,
\item[\textup{(3)}] The height $h(u_n)$ is bounded independently of $n$.   
\end{enumerate}  
\end{lemma}
\begin{proof}[(Proof of Lemma \ref{lemma:decomp})] By assumptions (a) and (b) on $S=S(\phi,b)$, we see that $\phi^n(b)\in\mathcal{O}_{K,S}$ for all $n$. Hence, for any integer $\ell\leq2$, we may write $\phi^n(b)=u_n\cdot d_n\cdot y_n^\ell$ as on (\ref{decomp})
since $\mathcal{O}_{K,S}$ is a UFD. Furthermore, we can assume that $0\leq v(d_n)\leq\ell-1$ for all $v\notin S$. To see this, we use the correspondence $V_K\mysetminus S\longleftrightarrow\Spec(\mathcal{O}_{K,S})$ discussed in \cite[Ch. 14]{Rosen} and write:
\[d_n=p_1^{e_1}\cdot p_2^{e_2}\cdots p_s^{e_s}\big(p_1^{q_1}\cdot p_2^{q_2}\cdots p_s^{q_s}\big)^{\ell},\;\;\;\;\; p_i\in\Spec(\mathcal{O}_{K,S})\]
for some integers $e_i, q_i$ satisfying $v_{p_i}(d_n)=q_i\cdot\ell+e_i$ and $0\leq e_i<\ell$. In particular, by replacing $d_n$ with $\big(p_1^{e_1}\cdot p_2^{e_2}\cdots p_s^{e_s}\big)$ and $y_n$ with $\big(y_n\cdot p_1^{q_1}\cdot p_2^{q_2}\cdots p_s^{q_s}\big)$, we may assume that $0\leq v(d_n)\leq\ell-1$ for all $v\in V_K\mysetminus S$ as claimed.
 
On the other hand, let $\mathfrak{p}_i\in S$ when $K$ is a number field and  $\mathfrak{p}_i\in S\mysetminus\{\mathfrak{p}_0\}$ when $K$ is a function field. Since the class group of $K$ is finite, there exists $a_i\in \mathcal{O}_K$ and $n_i\geq 1$ such that $\mathfrak{p}_i^{n_i}=(a_i)$. In particular, $v_{\mathfrak{p}_i}(a_i)=n_i>0$ and $v_\mathfrak{p}(a_i)=0$ for all $\mathfrak{p}\in\Spec(\mathcal{O}_K)\mysetminus\{\mathfrak{p}_i\}$. Therefore, if we write $v_{\mathfrak{p}_i}(d_n)=q_i\cdot n_i+r_i$ for some $0\leq r_i<n_i$ and set $d_n':=d_n/(\prod_{i}a_i^{q_i})$, then we have that $0\leq v_{\mathfrak{p}}(d_n')=v_{\mathfrak{p}}(d_n)\leq\ell-1$ for all $\mathfrak{p}\not\in S$ and $v_{\mathfrak{p}_i}(d_n')=r_i$ for all $\mathfrak{p}_i$. Hence, after replacing $d_n$ with $d_n'$ and $u_n$ with $u_n\cdot(\prod_{i}a_i^{q_i})$, we see that conditions (1) and (2) of Lemma \ref{lemma:decomp} are satisfied for $r(S):=\max\{r_i\}$.   
 
Finally, since $\mathcal{O}_{K,S}^*$ is a finitely generated group (\cite[Prop. 14.2 ]{Rosen}), we can absorb $\ell$-powers into $y_n$ and write $u_n=\textbf{u}_1^{r_1}\cdot \textbf{u}_2^{r_2}\dots \textbf{u}_t^{r_t}$ for some basis $\{\textbf{u}_i\}$ of $\mathcal{O}_{K,S}^*$ and some integers $0\leq r_i\leq\ell-1$. In particular, we assume that the height $h(u_n)$ is bounded independently of $n\geq0$.   
\end{proof}
\textbf{It is our goal to show that} $\mathbf{d_n\in\mathcal{O}_K}$ \textbf{contains primitive prime divisors outside of} $\mathbf{S}$. To do this, first note that conditions (b) and (c) imply that $\phi$ has good reduction (see \cite[Theorem. 2.15]{Silv-Dyn}) modulo the primes in $V_K\mysetminus S$. In particular, if $\mathfrak{p}$ is a prime of $\mathcal{O}_{K,S}$ such that $v_{\mathfrak{p}}(d_n)>0$ and $v_{\mathfrak{p}}(\phi^m(b))>0$ for some $1\leq m\leq n-1$ and $\phi^m(b)\neq0$, then 
\begin{equation}{\label{congruence}} \phi^{n-m}(0)\equiv\phi^{n-m}(\phi^{m}(b))\equiv\phi^n(b)\equiv0\Mod{\mathfrak{p}};
\end{equation}
see \cite[Theorem. 2.18]{Silv-Dyn}.
Therefore, if $n\geq1$ is such that $d_n$ has no primitive prime divisors outside os $S$, then we have the refined  factorization of ideals in $\mathcal{O}_{K,S}$: 
\begin{equation}{\label{refinement}}  (d_n)=\prod \mathfrak{p}_j^{e_j},\;\;\text{where}\;\; \mathfrak{p}_j\big\vert\phi^{t_j}(b)\;\text{or}\; \mathfrak{p}_j\big\vert\phi^{t_j}(0)\;\; \text{for some}\; 1\leq t_j\leq\Big\lfloor \frac{n}{2}\Big\rfloor. 
\end{equation}
Moreover, as noted above, we may assume that $0\leq e_j\leq\ell-1$. Hence, Lemma \ref{lemma:decomp} and Lemma \ref{htsbydiv} imply that there is a constant $c(K,S)$ such that 
\begin{equation}{\label{htestimate}} \boxed{h(d_n)\leq(\ell-1)\bigg(\sum_{i=1}^{\lfloor\frac{n}{2}\rfloor}h(\phi^i(b))+ \sum_{j=1}^{\lfloor\frac{n}{2}\rfloor}h(\phi^j(0))\bigg)+c(K,S)}; 
\end{equation}
here, we use that $\phi^j(0)\neq0$ for all $j\geq1$. Now, choose an integer $m=m_\phi$  and $\ell=\ell_\phi$ such that 
\[C:=C_{\ell,m}(\phi): Y^\ell=\phi^{m}(X)\] is a nonsingular curve of genus at least two. This is possible in the function field case since $\phi$ is dynamically $\ell$-power non-isotrivial. As for the number field case, since zero is not periodic, the Riemann-Hurwitz formula implies that $\#\phi^{-m}(0)\geq d^{m-2}$ for all $m\geq1$; see \cite[Exercise 3.37]{Silv-Dyn}. Hence, we may choose $m$ such that $\#\phi^{-m}(0)\geq5$. Now choose $\ell$ coprime to the multiplicities of all roots of $\phi^m$, and apply \cite[Corollary 2.2]{Rafe} or \cite[Proposition 3.7.3]{ffields}.
 
If $n\leq m_\phi$ for all $n\in\mathcal{Z}(\phi,b)$ with $b\in\mathcal{O}_{K,S}$ and $\hat{h}_\phi(b)>0$, then we are done. Otherwise, we may assume that $n>m_\phi$ so that (\ref{decomp}) implies that
\[P_n(b):=\big(\phi^{n-m_\phi}(b)\;,\;y_n\cdot\mysqrt{-2}{1}{\ell}{u_n\cdot d_n}\, \big)\in C(\widebar{K})\;\;. \]
It follows from the Vojta Conjecture \cite[Conj. 24.1\,\text{or}\, 25.1]{Vojta} for number fields or any of the bounds (suitable to positive characteristic) discussed in the introductions of \cite{Kim,htineq} for function fields, that there are non-negative constants $A_1=A_1(d,\ell_\phi,m_\phi)$ and  $A_2=A_2(\phi,d,\ell_\phi,m_\phi)$ such that the following inequality holds
\begin{equation}{\label{Szpiro}}
h_{\kappa(C)}(P_n(b))\leq A_1\cdot \mathfrak{d}(P_n(b))+A_2;
\end{equation} 
here $\kappa(C)$ is a canonical divisor class of $C$ with associated height function $h_{\kappa(C)}:C\rightarrow\mathbb{R}_{\geq0}$ (see \cite[\S2.2]{Voloch}), $\mathfrak{d}(P_n(b))$ is the logarithmic discriminant of $K(P_n(b))/K$ relative to $K$ (see \cite[\S 23]{Vojta}) over number fields, and
\[\mathfrak{d}(P_n(b)):=\frac{2\cdot \genus\big(K_n(b)\big)-2}{\big[K_n(b):K\big]}\;\;\;\;\text{for}\;\;\;K_n(b):=K\big(\mysqrt{-2}{1}{\ell}{u_n\cdot d_n}\,\big)\]
over function fields. We note that the bounds on (\ref{Szpiro}) have been made more explicit in \cite[\S2.2]{Voloch} over function fields, although we do not need them here.
 
On the other hand, if $K$ is a function field, then it follows from \cite[Prop. 3.7.3]{ffields} and the remark \cite[Remark 3.7.5]{ffields} that there is a constant $B_1=B_1(\mathfrak{g}_K)$, depending only on the genus $\mathfrak{g}_K$ of $K$, such that $\mathfrak{d}(P_n(b))\leq h(u_n\cdot d_n)+B_1$. Likewise, $h(u_n\cdot d_n)\leq h(d_n)+B(K,S)$, since the height of $u_n$ is absolutely bounded. Hence, there is a constant $B(K,S)$, depending only on $K$ and $S$, such that $\mathfrak{d}(P_n(b))\leq h(d_n)+B(K,S)$. Similarly, $\mathfrak{d}(P_n(b))\leq h(d_n)+B(K,S)$ over number fields; see \cite[\S 23]{Vojta}. In either case, we deduce from (\ref{Szpiro}) that
\begin{equation}{\label{bdwodisc}} h_{\kappa(C)}(P_n(b))\leq A_1\cdot h(d_n)+A_3
\end{equation}
where $A_3=A_3(\phi,d,m_\phi,\ell_\phi,\mathfrak{g}_K,S)=(B_1(\mathfrak{g}_K)+B_2(K,S))\cdot A_1(d,m_\phi)+A_2(\phi,d,\ell_\phi,m_\phi)$. In particular, all constants are independent of the basepoint. However, we want a bound relating $h(\phi^{n-m}(b))$ and $h(d_n)$. To do this, we note that if $\mathcal{D}_1$ is an ample divisor on $C$ and $\mathcal{D}_2$ is an arbitrary divisor, then
\begin{equation}{\label{Divisor}} \lim_{h_{\mathcal{D}_1}(P)\rightarrow\infty}\frac{h_{\mathcal{D}_2}(P)}{h_{\mathcal{D}_1}(P)}=\frac{\deg\mathcal{D}_2}{\deg{\mathcal{D}_1}},\;\;\;\;\;\;P\in C(\widebar{K});
\end{equation} 
see \cite[Thm III.10.2]{SilvA}. In particular, if $\pi:C\rightarrow\mathbb{P}^1$ is the covering $\pi(X,Y)=X$, then a degree one divisor on $\mathbb{P}^1$ (giving the usual height $h_K$ on projective space) pulls back to a $\deg(\pi)$ divisor $\mathcal{D}_2$ on $C$ satisfying $h_{\mathcal{D}_2}(P)=h(\pi(P))$. We deduce from (\ref{Divisor}) that there exists a constant $\delta=\delta(\phi,m,\ell)$ satisfying: 
\begin{equation}{\label{limit}} h_{\kappa(C)}(P)>\delta\;\;\;\;\;\text{implies}\;\;\;\;\; h(\pi(P))\leq \frac{\ell}{2\mathfrak{g}_C-2}\cdot h_{\kappa(C)}(P)+1\leq \frac{\ell}{2}\cdot h_{\kappa(C)}(P)+1
\end{equation}
for all $P\in C(\widebar{K})$; here $\mathfrak{g}_C\geq2$ is the genus of $C$, and $\deg(\pi)\leq\ell$. As an alternative to the height ratio on (\ref{Divisor}) over function fields, we could use explicit calculations of $h_{\kappa(C)}$ in \cite[\S4]{Voloch}; however, we prefer a uniform approach over global fields when possible. On the other hand, we note that the set of points $\{P_n(b)\}\subseteq C(\widebar{K})$ where (\ref{limit}) fails, 
\[T_C:=\big\{P_n(b)\;\big\vert\; h_{\kappa(C)}(P_n(b))\leq\delta\}\subseteq\big\{P\in C(\bar{K})\;\big\vert\;\; h_{\kappa(C)}(P)\leq\beta\;\;\,\text{and}\,\;\; [K(P):K]\leq\ell\big\},\]
is finite, since the canonical class $\kappa(C)$ is ample in genus at least two; see \cite[Thm. 10.3]{SilvA}. In particular, we deduce that $P_n(b)\in T_C$ implies $h(\phi^{n-m}(b))$ is bounded. Hence, Lemma \ref{0orbit} implies that $n$ is bounded independently of $b$ as claimed. Conversely, if $P_n(b)\not\in T_C$, then (\ref{htestimate}), (\ref{bdwodisc}) and (\ref{limit}) imply that 
\begin{equation}{\label{bdwocan}}
h(\phi^{n-m_\phi}(b))\leq\frac{\ell(\ell-1)}{2}\cdot A_1\cdot \bigg(\sum_{i=1}^{\lfloor\frac{n}{2}\rfloor}h(\phi^i(b))+ \sum_{j=1}^{\lfloor\frac{n}{2}\rfloor}h(\phi^j(0))\bigg)+A_4,
\end{equation}
for $A_4=A_4(\phi,d,\ell_\phi,m_\phi,\mathfrak{g}_K,K,S)=\ell/2\cdot (A_1\cdot c(K,S)+A_3)+1$. However, for wandering basepoints the left hand side of (\ref{bdwocan}) grows like $d^{n-m}$, while the right hand side of (\ref{bdwocan}) grows like $d^{\lfloor\frac{n}{2}\rfloor+1}$. In particular,  since $m$ is fixed, $n$ is bounded. To make this formal, we use properties of $\hat{h}_\phi$, the canonical height function attached to $\phi$. Specifically, it is known that:
\begin{equation}{\label{standard}}\text{(a).}\;\;\;\hat{h}_\phi=h+O(1)\;\;\;\;\;\;\;\;\;\;\;\; \text{(b).}\;\;\;\hat{h}_\phi(\phi^s(\alpha))=d^s\cdot \hat{h}_\phi(\alpha)
\end{equation}
for all $\alpha\in \widebar{K}$ and all integers $s\geq0$; see \cite[Thm. 3.20]{Silv-Dyn}. In particular, we deduce from (\ref{bdwocan}) that
\begin{equation}{\label{main}} d^{n-m_\phi}\cdot\hat{h}_\phi(b)\leq\bigg( \frac{\ell(\ell-1)}{2} \big(\hat{h}_\phi(b)+\hat{h}_\phi(0)\big)A_1\bigg)\frac{d^{\lfloor\frac{n}{2}\rfloor+1}-1}{d-1}+\bigg(\frac{\ell(\ell-1)}{2}A_1 B_\phi\bigg)n+A_5;
\end{equation}
here $|\hat{h}_\phi-h|\leq B_\phi$ from (\ref{standard}) and $A_5(\phi,d,\ell_\phi,m_\phi,\mathfrak{g}_K,K,S)=A_4+B_\phi$. Moreover, since $\hat{h}_\phi(b)\neq0$ and $\hat{h}_{\phi,K}^{\min}$ is positive (see the proof Lemma \ref{0orbit}), it follows that
\begin{equation} d^{n-m_\phi}\leq B_4\cdot d^{\lfloor\frac{n}{2}\rfloor+1}+B_5\cdot n+B_6,
\end{equation}
where the constants $B_4$, $B_5$, and $B_6$ are all independent of the basepoint.   
However, such an inequality implies for instance that
\begin{equation}
n\leq 5+2m_\phi+2\log_d\big(B_{\max}\big)
\end{equation} 
where $B_{\max}:=\max\{B_4,B_5,B_6\}$. Hence, we see that $n$ is bounded, or equivalently 
\[\sup\big\{n\;\big|\; n\in\mathcal{Z}(\phi,b)\;\,\text{for some}\;\, \hat{h}_{\phi}(b)>0,\; b\in\mathcal{O}_{K,S}\big\}\]
is finite. Therefore, when $0\not\in \mathcal{O}_\phi(b)$, we have shown that $\#\mathcal{Z}(\phi,S)$ is bounded by a constant that depends only on $\phi$, $S$, and $K$. On the other hand, suppose that $0\in\mathcal{O}_\phi(b)$. Lemma \ref{0orbit} implies that there exists $n_\phi$ (depending only on $\phi$ and $K$) such that $\phi^n(a)\neq0$ for all $n\geq n_\phi$ and all $a\in K$. It follows that $\mathfrak{Z}_{\phi,K}:=\big\{a\in K: 0\in\mathcal{O}_\phi(a)\big\}$ is a finite set. In particular, the set of primes
\[S_0:=\big\{\text{primes}\;\mathfrak{p}\in V_K\,:\, v_\mathfrak{p}(\phi^n(a))>0\,\;\text{for some}\;a\in\mathfrak{Z}_{\phi,K},\; n\leq n_\phi,\; \phi^n(a)\neq0\big\}\]
is also finite. Let $S'=S_0\cup S$, where $S$ satisfies conditions (a)-(d) above. We have shown (in the first case of Theorem \ref{thm:Avg}) that there exists $N(\phi,S')$ such that $\phi^n(c)$ contains primitive prime divisors outside of $S'$ for all $n\geq N(\phi,S')$ and all $c\in\mathcal{O}_{K,S}\mysetminus\mathfrak{Z}_{\phi,K}$. We apply this to $c:=\phi^{n_\phi}(b)$, from which it follows that $n\in\mathcal{Z}(\phi,b)$ implies $n\leq n_\phi+N(\phi,S')$. As in the previous case, this bound only depends on $\phi$, $S$ and $K$.  
 
We note that the use of the Vojta conjecture in the number field setting is not necessary to prove the finiteness of $\mathcal{Z}(\phi,S)$ when zero is preperiodic. To see this, note that Lemma \ref{lemma:decomp} and (\ref{congruence}) imply that $\{d_n\}_{n\in\mathcal{Z}(\phi,S)}$ is finite, since there are only finitely many primes not in $S$ dividing elements of the orbit of zero (equivalently $d_n$ has bounded height by Lemma \ref{htsbydiv}). Now define \[K(\phi,S):=K\big(\mysqrt{-2}{1}{\ell}{u_n\cdot d_n}\;:\;n\in\mathcal{Z}(\phi,S)\big),\]
a finite extension of $K$. Hence, Faltings' Theorem \cite{Faltings} implies that the set of $K(\phi,S)$ rational points of $C$ is finite. In particular, $\phi^{n-m_\phi}(b)$ has bounded height, and Lemma \ref{0orbit} implies that $n$ is bounded as claimed.       
 
As for statements about averages, we consider the set
\begin{equation}{\label{densityset}}
T_{\phi,n,S}:=\big\{b\in\mathcal{O}_{K,S}\;\vert\; n\in\mathcal{Z}(\phi,b),\;\,\hat{h}_\phi(b)\neq0\big\}.
\end{equation}
It follows from the fact that $\mathcal{Z}(\phi,S)$ is finite that $T_{\phi,n,S}=\varnothing$ for all $n$ sufficiently large. Although it may be the case that $T_{\phi,n,S}$ is infinite for some $n$, we will show that $T_{\phi,n,S}$ is always a sparse set $\mathcal{O}_{K,S}$. With this in mind, for any subset $E\subseteq \mathcal{O}_{K,S}$ we define the (natural) upper density $\widebar{\delta}_{\Nat,S}(E)$ of $E$ to be the quantity
\begin{equation} \widebar{\delta}_{\Nat,S}(E):=\limsup_{B\rightarrow\infty}\frac{\#\{b\in E\;|\; h(b)\leq B\}}{\#\{b\in\mathcal{O}_{K,S}\;|\; h(b)\leq B\}}:=\limsup_{B\rightarrow\infty}\frac{\#E(B)}{\#\mathcal{O}_{K,S}(B)}.
\end{equation}
In particular, if we set $N(\phi,S):=\sup\{n: n\in\mathcal{Z}(\phi,S)\}$, then we see that
\begin{equation}{\label{estimate}} \widebar{\Avg}(\mathcal{Z}(\phi),S)\leq N(\phi,S)\cdot\bigg(\sum_{n=1}^{N(\phi,S)}\widebar{\delta}_{\Nat,S}\big(T_{\phi,n,S}\big)\bigg).
\end{equation} 
Therefore, it suffices to prove that $\widebar{\delta}_{\Nat,S}(T_{\phi,n,S})=0$ for all $1\leq n\leq N(\phi,S)$, to deduce that $\widebar{\Avg}(\mathcal{Z}(\phi),S)=0$. To do this, we make a few auxiliary definitions: for all rational functions $f\in K(x)$ and all sets of primes $\mathcal{P}\subseteq V_K$, we define:
\begin{equation}{\label{suppset}}
I_{f,S,\mathcal{P}}:=\big\{b\in\mathcal{O}_{K,S}\;:\; \Supp(f(b))\subseteq\mathcal{P}\big\}.
\end{equation}
Here, for any $\alpha\in K$, the support $\Supp(\alpha)$ is the set of all primes $\mathfrak{p}$ such that $v_\mathfrak{p}(\alpha)>0$. Furthermore, let $\mathcal{P}_{0}$ be the finite set of prime divisors of the first $N(\phi,S)$-elements of the orbit of zero, that is $\mathcal{P}_{0}:=\{\Supp(\phi^n(0))\}_{n\leq N(\phi,S)}$. Finally, let $\mathcal{P}_f$ be the set of primes of bad reduction of $f$.
 
In particular, it follows from (\ref{congruence}) that    
\begin{equation}{\label{subset}}
T_{\phi,n,S}\subseteq I_{\phi^n,S,\mathcal{P}} \;\;\, \text{for}\,\; \mathcal{P}=S\cup\mathcal{P}_{0}\cup\mathcal{P}_f\,.
\end{equation}
Therefore, if we let $f:=\phi^n$ for any $n\leq N(\phi,S)$, then our average-zero result follows from Lemma \ref{lemma:density} below. However, because of its possible independent interest, we state Lemma \ref{lemma:density} for subsets $\mathcal{P}\subseteq V_K$ of (Dirichlet) density zero, not just finite subsets; see \cite{Cheba} and \cite[\S3]{Serre} for the relevant background and results on densities over global fields. Moreover, in what follows $\N(\mathfrak{p})=\#k_\mathfrak{p}$ is the size of the residue field.
\end{proof} 
\begin{lemma}{\label{lemma:density}} Let $f:\mathbb{P}^1\rightarrow\mathbb{P}^1$ be non-constant rational function, let $\mathcal{P}\subseteq V_K$, and let
\begin{equation}{\label{Dirchlet}}
\widebar{\delta}_{\Dir}(\mathcal{P}):=\lim_{s\rightarrow 1^+}\frac{\sum_{\mathfrak{q}\in\mathcal{P}}\,\N(\mathfrak{q})^{-s}}{\sum_{\mathfrak{q}\in V_K}\N(\mathfrak{q})^{-s}}=\lim_{s\rightarrow 1^+}\frac{\sum_{\mathfrak{q}\in\mathcal{P}}\,\N(\mathfrak{q})^{-s}}{\log(1/(s-1))}
\end{equation}
be the Dirichlet density of $\mathcal{P}$. If $\widebar{\delta}_{\Dir}(\mathcal{P})=0$, then $\widebar{\delta}_{\Nat,S}(I_{f,S,\mathcal{P}})=0$.
\end{lemma}
\begin{proof}[Proof of Lemma \ref{lemma:density}] Since, $I_{f,S,\mathcal{P}_1}\subseteq I_{f,S,\mathcal{P}_2}$ whenever $\mathcal{P}_1\subseteq \mathcal{P}_2$, we may enlarge $\mathcal{P}$ and assume that $\mathcal{P}$ contains both $S$ and $\mathcal{P}_f$. Now let $\mathcal{P}'=\{\mathfrak{p}\notin\mathcal{P}: f\;\text{has a root}\Mod{\mathfrak{p}}\}$, and for $\mathfrak{p}\in\mathcal{P}'$ let $a_\mathfrak{p}$ be such a root$\Mod{\mathfrak{p}}$; note that this makes sense, i.e $f:\mathbb{P}^1(\mathbb{F}_\mathfrak{p})\rightarrow\mathbb{P}^1(\mathbb{F}_\mathfrak{p})$ is well defined, since $\mathfrak{p}\notin\mathcal{P}_f$. It follows from the Chebotarev Density Theorem, that $\delta_{\Dir}(\mathcal{P}')$ is positive. Let $\mathcal{P}''$ be any finite subset of $\mathcal{P}'$. By definition of $I_{f,S,\mathcal{P}}$, we see that
\begin{equation}{\label{sieve}} I_{f,S,\mathcal{P}}\subseteq \big\{b\in\mathcal{O}_{K,S}\,:\, b\not\equiv{a_p}\Mod{\mathfrak{p}} \,\;\text{for all}\;\mathfrak{p}\in\mathcal{P}''\big\}.
\end{equation}
Fix $\mathfrak{p}\in\mathcal{P}''$. Since $\delta_{\Nat,S}$ is translation invariant, $\delta_{\Nat,S}(a+\mathfrak{p})$ is independent of $a\in\mathcal{O}_{K,S}$. In particular, we add up $\delta_{\Nat,S}(a+\mathfrak{p})$ over representatives $a\in\mathcal{O}_{K,S}/\mathfrak{p}\mathcal{O}_{K,S}$ and see that the natural density of $\{b\in\mathcal{O}_{K,S}\,:\; b\equiv{a_p}\Mod{\mathfrak{p}}\}$ is $1/\N(\mathfrak{p})$ as expected. In particular, one computes via the Chinese Remainder Theorem that the natural density of the set displayed on the right hand side of (\ref{sieve}) is $\prod_{\mathfrak{p}\in\mathcal{P}''}\big(1-\frac{1\,}{\N(\mathfrak{p})}\big)$. On the other hand,
\begin{equation} \widebar{\delta}_{\Nat,S}(I_{f,S,\mathcal{P}})\leq \prod_{\mathfrak{p}\in\mathcal{P}',\; \N(\mathfrak{p})\leq B}\Big(1-\frac{1\,}{\N(\mathfrak{p})}\Big)\widesim{} C/\log(B)^{\delta_{\Dir}(\mathcal{P}')}
\end{equation}
as $B\rightarrow\infty$ for some positive constant $C$; see \cite[Exercise 3.3.2.2]{Serre}. In particular, since the density $\delta_{\Dir}(\mathcal{P}')\neq0$, we see that $\widebar{\delta}_{\Nat,S}(I_{f,S,\mathcal{P}})=0$ as claimed.        
\end{proof}  
\begin{lemma}{\label{0orbit}} Let $B>0$. There exists a positive integer $n_\phi(B)$ such that $h(\phi^n(b))\leq B$ implies $n<n_\phi(B)$\; for all $b\in K$ satisfying $\hat{h}_\phi(b)\neq0$.
\end{lemma}
\begin{proof}[(Proof of Lemma \ref{0orbit})] Suppose that $b\in K$ and $h(\phi^n(b))\leq B$. Since $\hat{h}_\phi=h+O(1)$ and $\hat{h}_\phi(\phi^n(b))=d^{n}\cdot\hat{h}_{\phi}(b)$, we see that $d^n\cdot\hat{h}_\phi(b)=\hat{h}_\phi(\phi^n(b))\leq B'$ for some positive constant $B'$ depending of $\phi$ and $B$. Moreover,
\begin{equation}{\label{hatmin}} \hat{h}^{\min}_{\phi,K}:=\inf\big\{\hat{h}_\phi(c)\;|\: c\in\mathbb{P}^1(K),\,\;\hat{h}_\phi(c)>0\big\}
\end{equation} is strictly positive. To see this, choose an arbitrary wandering point $c_0\in \mathbb{P}^1(K)$ for $\phi$ (possible, for instance, by Northcott's Theorem \cite[Theorem. 3.12]{Silv-Dyn}), and note that
\[\hat{h}^{\min}_{\phi,K}=\inf\big\{\hat{h}_\phi(c)\;|\: c\in\mathbb{P}^1(K)\;\text{and}\; 0<\hat{h}_\phi(c)<\hat{h}_\phi(c_0)\big\}.\]
However, this latter set is finite and consists of strictly positive numbers; hence $\hat{h}^{\min}_{\phi,K}> 0$. In particular, it follows that $h_K(\phi^n(b))\leq B$ implies that $n\leq\log_d\big(B'/\hat{h}^{\min}_{\phi,K}\big)$ as claimed.   
\end{proof}
\begin{lemma}{\label{htsbydiv}} If $K$ is one of the global fields in Theorem \ref{thm:Avg}, then
\begin{equation}\label{lemma:htsbydiv} \sum_{\mathfrak{p}\subseteq\mathcal{O}_{K},v_\mathfrak{p}(\alpha)\geq0}v_\mathfrak{p}(\alpha)\cdot\N_{\mathfrak{p}}=h(\alpha)
\end{equation}
for all nonzero $\alpha\in\mathcal{O}_{K}$. On the other hand, $\sum_{v_\mathfrak{p}(\alpha)\geq0}v_\mathfrak{p}(\alpha)\N_{\mathfrak{p}}\leq h(\alpha)$ for all $\alpha\in K^*$.      
\end{lemma}
\begin{proof} Let $K$ be a function field. If $\alpha\in\mathcal{O}_K=\mathcal{O}_{\mathfrak{p}_0}$ is non-constant, then $v_{\mathfrak{p}_0}(\alpha)<0$, and the claim follows from (\ref{ffhtdef}), i.e. the number of zeros equals the number of poles (when counted with the correct multiplicity). On the other hand, let $K/\mathbb{Q}$ be a number field. Then the product formula implies
\begin{equation}{\label{alt}}
h(\alpha)=\sum_{\mathfrak{p}\subseteq\mathcal{O}_{K}, v_\mathfrak{p}(\alpha)\geq0}v_{\mathfrak{p}}(\alpha)\cdot\N_{\mathfrak{p}} \;-\; \frac{1}{[K:\mathbb{Q}]}\sum_{\sigma:K\rightarrow\mathbb{C}}\min(\log|\sigma(\alpha)|,0).
\end{equation}   
In particular, if $K=\mathbb{Q}$ or $K/\mathbb{Q}$ is an imaginary quadratic extension, then one verifies directly that $|\sigma(\alpha)|>1$ for all $\sigma:K\rightarrow\mathbb{C}$ and all non-zero $\alpha\in\mathcal{O}_K\mysetminus\,\mathcal{O}_K^*$. Therefore, (\ref{alt}) implies the claim for such integers. Conversely, if $\alpha\in\mathcal{O}_K^*$, then $\alpha$ is a root of unity and $h(\alpha)=0$. On the other hand, if $\alpha$ is a unit, then $v_{\mathfrak{p}}(\alpha)=0$ for all $\mathfrak{p}$ and (\ref{lemma:htsbydiv}) holds. In either setting, we see that the inequality, $\sum_{v_\mathfrak{p}(\alpha)\geq0}v_\mathfrak{p}(\alpha)\N_{\mathfrak{p}}\leq h(\alpha)$ for all $\alpha\in K^*$, follows from the product formula.        
\end{proof}
\begin{remark} We note that Theorem \ref{thm:Avg} part 1(a) is a strengthening of the main result of \cite{Rice}. Moreover, for results on $\mathcal{Z}(\phi,b)$ when $0\in\Per(\phi)$, see \cite{Ing-Silv}.
\end{remark}
\begin{remark} In characteristic zero, the finiteness of $\mathcal{Z}(\phi,S)$ holds for $K=k(t)$ with essentially the same proof: use the fact that $\mathcal{O}_{K,S}$ is unique factorization domain for all $S$ and that $\hat{h}_{\phi,K}^{\min}$ is positive \cite[Remark 1.7(ii)]{Baker}. On the other hand, the proof of Theorem \ref{thm:Avg} breaks down when $K\neq k(t)$, since the class group of $K$ is not finite. Note also that $\widebar{\Avg}(\mathcal{Z}(\phi),S)$ does not make sense for characteristic zero function fields, since the Northcott property fails.        
\end{remark}
\begin{remark} For number fields $K/\mathbb{Q}$, the height calculation in Lemma \ref{htsbydiv} fails whenever the group of units $\mathcal{O}_K^*$ is infinite. In particular, one cannot in general calculate heights of algebraic integers by solely keeping track of divisors. Therefore, in order to generalize Theorem \ref{thm:Avg} to all number fields, one would need to estimate $|\phi^n(b)|_{\sigma}$ for the archimedean places $\sigma:K\rightarrow\mathbb{C}$ as well. 
\end{remark}
\begin{remark} Of course, one would like to know whether Theorem \ref{thm:Avg} holds for rational functions. However, in Lemma \ref{lemma:decomp} and elsewhere in the proof of Theorem \ref{thm:Avg}, we used that $\phi^n(b)\in\mathcal{O}_{K,S}$ for all $n$, a property that will fail in general. For instance, over number fields Silverman has shown that $\phi^2\not\in\widebar{K}[x]$ implies $\mathcal{O}_{\phi}(b)\cap\mathcal{O}_{K,S}$ is finite for all $b\in K$; see \cite{Silv-Duke} for Silverman's integral point theorem, and see \cite{Me-Avg} for an average-version.    
\end{remark}
By analogy with the function field case, we conjecture that Theorem \ref{thm:Avg} holds over all number fields without assumptions on the orbit of zero.
\begin{conjecture}{\label{numflds}} Let $K/\mathbb{Q}$ and $\phi\in K[x]$ be a polynomial of degree $d\geq2$. If $\phi(x)\neq c\cdot x^d$, then $\mathcal{Z}(\phi,S)$ is finite and $\widebar{\Avg}(\mathcal{Z}(\phi),S)=0$\, for all finite subsets $S\subseteq V_K$.    
\end{conjecture}
As for the key condition over function fields, we expect that most rational functions $\phi\in K(x)$ are dynamically $\ell$-power non-isotrivial for some $\ell\geq2$; see Definition \ref{def:isotriv} above. In fact, it is likely that one can choose infinitely such exponents.
\begin{conjecture}{\label{conj:isotriv}} Suppose that $\phi\in K[x]\mysetminus \widebar{\mathbb{F}}_p[x]$ satisfies the following conditions:
\begin{enumerate}
\item[\textup{(1)}] $\deg(\phi)\geq 2$,
\item[\textup{(2)}]  $\gcd(\deg(\phi),p)=1$,
\item[\textup{(3)}]  $\phi(x)\neq c\cdot x^d$ for all $c\in\widebar{\mathbb{F}}_p$\,.
\end{enumerate}
Then there exists $\ell\geq2$ and $m\geq1$, such that $\gcd(\ell,p)=1$ and $C_{\ell,m}(\phi): Y^\ell=\phi^m(X)$ is a non-isotrivial curve of genus at least $2$.  
\end{conjecture}
However, at present, Conjecture \ref{conj:isotriv} seems quite difficult. On the other hand, there are formulas for the relevant Kodaira-Spencer maps in \cite{Voloch} in terms of the iterated preimages of zero. It is therefore possible that one can exploit knowledge of $\Gal_K(\phi^m)$ for some (hopefully small) $m\geq1$, to show that the KS map is non-zero.
 
We carry out these calculations for the polynomials $\phi(x)=x^d+f$, where the KS map computation becomes a sum over cyclotomic characters. Although special cases, these polynomials are important examples in several ways. First of all, the curves $C_{\ell,1}(\phi)$ are isotrivial for all $\ell\geq2$ while $C_{d,2}(\phi)$ and $C_{2,2}(\phi)$ are not (see \ref{eg:isotriv}), illustrating that one must in general pass to an iterate when studying primitive prime divisors. Secondly, we can use Theorem \ref{thm:Avg} to show that the Galois groups of iterates of $\phi(x)=x^d+f$ form a finite index subgroup of an infinite iterated wreath product of cyclic groups. In particular, they provide the first examples of a dynamical Serre-type open image theorem over global fields; compare to results for quadratic rational maps over number fields in \cite{ABCimplies,Rafe-Manes,Stoll-Galois}.      
\begin{remark}{\label{remark:isotriv}} It is worth pointing out that we could just as well use the more standard notion of isotriviality in Definition \ref{def:isotriv}, Theorem \ref{thm:Avg} and Conjecture \ref{conj:isotriv}: a curve is said to be isotrivial (in the standard sense) if after a base extension it may be defined over a finite field; see the Appendix in \cite{Voloch-Survey}. Strictly speaking, the key bounds in \cite{Kim,htineq} are for curves with non-zero Kodaira-Spencer class. However, the general case follows from this one as follows: assuming that $C_{\ell,m}(\phi)_{/K}$ is a non-isotrivial curve (in the the standard sense), there is an $r$ (a power of the characteristic of $K$) and a separable extension $L/K$ such that $C_{\ell,m}(\phi)$ is defined over $L^r$ and that the Kodaira-Spencer class of $C_{\ell,m}(\phi)$ over $L^r$ is non-zero. Now, if we apply any of the bounds in \cite{Kim,htineq} to $C_{\ell,m}(\phi)$ over $L^r$, then we achieve the bounds on (\ref{Szpiro}); the rest of the proof of Theorem \ref{thm:Avg} is the same. However, we prefer the more explicit (computational) condition that the Kodaira-Spencer map be non-zero. 
\end{remark}
On the other hand, given a quadratic polynomial $\phi(x)=(x-\gamma)^2+c\in K[x]$, the curve $C_{2,m}(\phi)$ for $m\geq2$ maps to the elliptic curve
\begin{equation}{\label{elliptic}} E_\phi: Y^2=(X-c)\cdot\phi(X)
\end{equation} 
via $(X,Y)\rightarrow\big(\phi^{m-1}(X), Y\cdot(\phi^{m-2}(X)-\gamma)\big)$. In particular, if the $j$-invariant \cite[III.1 Prop. 1.4]{Silv-Ellip} of $E_\phi$ is non-constant, then it follows from \cite[Proposition 3.3]{primdiv} that $C_{2,m}(\phi)$ is non-isotrivial (in the standard sense) for all $m\geq2$. Therefore, Remark \ref{remark:isotriv} implies an explicit form of Theorem \ref{thm:Avg} in the quadratic case:
\begin{corollary}{\label{cor:quad}} Let $K/\mathbb{F}_q(t)$ be a finite extension of odd characteristic. For all monic, quadratic polynomials $\phi\in K[x]$, write $\phi(x)=(x-\gamma)^2+c$ by completing the square. If $\phi(\gamma)\cdot \phi^2(\gamma)\neq 0$ and the quantity
\[\Big(\frac{27}{1728}\Big)\cdot j(E_\phi)=\frac{-\g^6 + 6\g^5c - 15\g^4c^2 + 9\g^4c + 20\g^3c^3 - 36\g^3c^2 + \dots + 27c^3}{\g^4c - 4\g^3c^2 + 6\g^2c^3 + 2\g^2c^2 - 4\g c^4 - 4\g c^3 + c^5 + 2c^4 + c^3} \]
is non-constant, then $\mathcal{Z}(\phi,S)$ is finite and $\widebar{\Avg}(\mathcal{Z}(\phi),S)=0$ for all $S$.  
\end{corollary}
See Corollary 1.2 and Corollary 1.4 of \cite{primdiv} for statements about the Galois groups $G_K(\phi)$ under the hypotheses of Corollary \ref{cor:quad} above.         
\end{section}
\begin{section}{3. Dynamical Galois groups and the Kodaira-Spencer map}
We now use our results on dynamical Zsigmondy sets to study dynamical Galois groups. To do so, we first introduce the necessary background material on wreath products, following the presentation in \cite[\S3.9]{Silv-Dyn} and the results in \cite{Juul}.      
\begin{definition}
Let $G$ be a group acting on an index set $A$, and let $H$ be an abelian group with its group law written additively. The set of maps $\Map(A,H)$ is naturally a group: for $i_1,i_2\in\Map(A,H)$, define  
\[ (i_1+1_2):H\rightarrow H,\;\;\;\; (i_1+i+2)(a)=i_1(a)+i_2(a).\]
Since $G$ acts on $A$, it comes equipped with an action on $\Map(A,H)$ as follows: 
\[g:\Map(A,H)\rightarrow\Map(A,H),\;\;\;\; g(i)(a)=i(g(a))\]
for all $g\in G$, $i\in\Map(A,H)$, and $a\in A$. The \emph{wreath product} of $G$ and $H$ (relative to $A$) is the set $\Map(A,H)\times G$ with the group law  
\[(g_1,i_1)\,*\,(g_2,i_2)=(g_2(i_1)+i_2\,,\,g_1g_2)\]
and is denoted $G[H]$.
\end{definition} 
\begin{definition} If $G$ acts on $A$, then $[G]^m$ acts on $A^m$ (the cartesian product) for all $m\geq1$. Therefore, we may define the \emph{$n$-th iterated wreath power of $G$} inductively: $[G]^1=G$ and $[G]^n=[G]^{n-1}[G]$.
\end{definition} 
\begin{definition}
Since $G[H]\rightarrow G$ via projection onto the second coordinate, we have a system of maps $[G]^n\rightarrow[G]^{n-1}$ allowing us to define an inverse limit. In the special case when $G=C_d$ is the cyclic group of order $d$ (acting on itself by translation), we define
\[W(d):=\lim_{\longleftarrow}\,[C_d]^n.\]
to be the infinite iterated wreath product of $C_d$; for more on $W(d)$, see \cite{wreath}.   
\end{definition}
Our primary interest in wreath products comes from their relationship to the Galois groups of compositions of rational functions. We restate the following result from \cite[Lemma 2.5]{Juul}.
\begin{lemma}{\label{lemma:borrowed}} Let $K$ be a field and let $\psi,\gamma\in K[x]$ with $\deg(\psi)=\ell$ and $\deg(\gamma)=d$. We assume that $\psi\circ\gamma$ has $\ell d$ distinct roots in $\widebar{K}$ and that $\psi$ is irreducible over $K$. Let $\alpha_1,\dots,\alpha_\ell$ be the roots of $\psi$, $M_i$ be the splitting field of $\gamma(x)-\alpha_i$ over $K(\alpha_i)$, and $G=\Gal_K(\psi)$. If $H=\Gal(M_i/K(\alpha_i))$, then there is an embedding $\Gal_K(\psi\circ\gamma)\leq G[H]$.         
\end{lemma}
As in the introduction, for $\phi\in K(x)$ and $n\geq1$, we let $K_n(\phi)$ be the field obtained by adjoining all solutions of $\phi^n(x)=0$ to $K$. Since $\phi$ has coefficients in $K$, the extension $K_n(\phi)/K$ is Galois, and we let $G_{K,n}(\phi):=\Gal_K(\phi^n)$ be the Galois group of $K_n(\phi)/K$. Since $K_{n-1}(\phi)\subseteq K_{n}(\phi)$ for all $n$ (with some separability assumptions), we may define $G_K(\phi)$ to be the inverse limit of the $G_{K,n}(\phi)$. Eschewing the generic situation for reasons of complexity, we focus our attention on the family of iterates of $\phi(x)=x^d+f$ for $f\in \mathbb{F}_p[t]$. We now restate and prove a Serre-type finite index Theorem for $G_K(\phi)$ from the introduction; this is the first (unconditional) result of this form over any global field in the non-quadratic case:               
\unicrit*
\begin{proof} To prove the first statement, it suffice to show that $\phi$ is dynamically $2$-power non-isotrivial; see Definition \ref{def:isotriv} and Theorem \ref{thm:Avg} above. When $d=2$, it suffice to calculate the $j$-invariant of the elliptic curve $E_\phi$ from (\ref{elliptic}) above. In particular, we compute that
\[\frac{1}{64}j(E_\phi)=\frac{-f^3 + 9f^2 - 27f + 27}{f^2+2f+1},\]
which cannot be constant unless $f$ is constant. Therefore, we may assume that $d\geq3$. If $d$ is odd, then $\{\frac{x^idx}{y}\}$ for $0\leq i\leq\frac{d^2-1}{2}-1$ is a basis for the space of regular $1$-forms \cite[Theorem 3]{differentials} on
\[C_{2,2}(\phi): Y^2=\phi^2(X)=(X^d+f)^d+f.\] 
To see that $C_{2,2}(\phi)$ is non-singular, use the discriminant formula in \cite[Lemma 2.6]{RJones}; there we see that the discriminant of $\phi^2$ is zero if an only if $\phi(0)\cdot\phi^2(0)=f\cdot(f^d+f)=0$. This is impossible, since $f$ is non-constant. 
 
We now calculate the Kodaira-Spencer map associated to the surface $C_{2,2}(\phi)\rightarrow\mathbb{P}^1$ using \cite{Voloch}. In keeping with the notation in \cite[\S5.2]{Voloch}, let $\{P_s\}$ for $1\leq s\leq d^2$ be the set of roots of $\phi^2$, and let $\phi_x^2$ and $\phi_t^2$ denote the partial derivatives of $\phi^2$ with respect to $x$ and $t$ respectively. It follows from Serre duality that the Kodaira-Spencer matrix with respect to the standard basis above is  \begin{equation}{\label{Matrix}}
m_{i,j}=\sum_s \frac{P_s^{i+j}\phi_t^2(P_s)}{2\phi_x^2(P_s)^2}; 
\end{equation}
see \cite[\S5.2]{Voloch}. One computes that $\phi_x^2(x)=d^2\cdot(\phi(x)\cdot x)^{d-1}$ and $\phi_t^2(x)=f'\cdot(d\cdot(\phi(x))^{d-1}+1)$. On the other hand, since $0=\phi^2(P_s)=\phi(\phi(P_s))$, we may write $\phi(P_s)=\zeta_s\cdot\alpha_f$ where $\alpha_f:=\sqrt[d]{-f}$ is a fixed $d$-th root of $-f$ in $\widebar{K}$ and $\zeta_s$ is a $d$-th root of unity. In particular, it follows from (\ref{Matrix}) that
\begin{equation}{\label{MatrixSimp}}
m_{i,j}=\bigg(\frac{f'}{2 d^4 f^2}\bigg)\cdot\sum_s P_s^{i+j-2d+2}\cdot\Big(\zeta_s^2\alpha_f^2-d\cdot\alpha_f\cdot f\cdot \zeta_s\Big).
\end{equation}
To show that the Kodaira-Spencer map is non-zero, it suffices to find a single entry $m_{i,j}\neq0$. To do this, let $i=d-2$ and $j=0$. In this case, we see from (\ref{MatrixSimp}) that
\[m_{d-2,0}=\bigg(\frac{f'}{2 d^4 f^2}\bigg)\cdot\sum_s\, \frac{1}{\zeta_s\alpha_f-f}\cdot\Big(\zeta_s^2\alpha_f^2-d\cdot\alpha_f\cdot f\cdot \zeta_s\Big)\] 
On the other hand, the formal identity $x^n-y^n=(x-y)\cdot(x^{n-1}+yx^{n-2}+\dots +y^{n-2}x+y^{n-1})$ applied to $n=d$, $x=\zeta_s\alpha_f$ and $y=f$ implies that
\[m_{d-2,0}=\!\bigg(\frac{f'}{2 d^4 f^2}\bigg)\,\sum_s\, \bigg(\!\bigg(\frac{-\alpha_f^{d-1}}{f^{d}+f}\bigg)\zeta_s^{d-1}+\bigg(\frac{-\alpha_f^{d-2}\cdot f}{f^{d}+f}\bigg)\zeta_s^{d-2}+\dots+\bigg(\frac{-f^{d-1}}{f^{d}+f}\bigg)\!\bigg)\cdot\Big(\zeta_s^2\alpha_f^2-d\cdot\alpha_f\cdot f\cdot \zeta_s\Big).\]
After regrouping terms and changing the order of summation, we see that
\[m_{d-2,0}=\bigg(\frac{f'}{2 d^4 f^2}\bigg)\!\cdot\!\bigg(\frac{(1-d)\cdot f}{f^{d-1}+1}+\sum_{k=1}^{d-1}\sum_sc_k\,\zeta_s^k \bigg)\]
for some constants $c_k$, depending only on $1\leq k\leq d-1$ (not on $s$). However, because the function sending $P_s\rightarrow \zeta_s$ is a $d:1$ surjection onto $\mu_d$ (the group of $d$-th roots of unity), the sum $\sum_sc_k\zeta^k=(d^2c_k)\sum_{\zeta\in\mu_d}\zeta^k=0$ for all indices $k$. We deduce that
\begin{equation}{\label{matrixformula}}
m_{d-2,0}= \frac{(1-d)\cdot f'}{2d^4(f^d+f)}.
\end{equation}
In particular, $m_{d-2,0}\neq0$ since $f\not\in K^p$ and $d\not\equiv1\Mod{p}$. It follows that $C_{2,2}(\phi)$ is a non-isotrivial curve of genus at least two, and Theorem \ref{thm:Avg} implies that $\mathcal{Z}(\phi,S)$ is finite and $\widebar{\Avg}(\mathcal{Z}(\phi),S)=0$\, for all finite subsets $S\subseteq V_K$ as claimed.
 
On the other hand, essentially the same argument shows that $C_{d,2}: Y^d=\phi^2(x)$ is a non-isotrivial curve of genus at least two, and we take this approach to prove Theorem \ref{thm:unicrit} part (1) when $d\geq4$ is even. To see that $C_{d,2}$ is non-isotrivial, we use the differentials $\frac{x^{d-2}dx}{y^{d-1}}$ and $\frac{dx}{y}$, which are both holomorphic by \cite[Theorem 3]{differentials}, to compute a non-zero entry of the Kodaira-Spencer matrix. In particular, the same Serre duality argument in \cite[\S5.2]{Voloch} implies that
\begin{equation}{\label{KSford}}
m_{(d-2,d-1),(0,1)}=\sum_s\Res_{P_s}\bigg(\frac{x^{d-2}\,\phi_t^2\, dx}{y^d\,\phi_x^2}\bigg)dt=\sum_s\,\frac{{P_s}^{d-2}\,\phi_t^2(P_s)\, dx}{d\cdot\phi_x^2(P_s)^2}dt,
\end{equation}
since $y^d=\phi^2(x)$, so that $\phi^2(x)/(x-P_s)|_{P_s}=\phi_x^2(P_s)$ and $(x-P_s)$ is of degree $d$; here $\Res$ denotes the residue map on the differentials of a curve \cite[Theorem 7.14.1]{Hart}. Hence, $m_{(d-2,d-1),(0,1)}$ is nothing but $2/d\cdot m_{d-2,0}$ from the hyperelliptic case on (\ref{Matrix}). We deduce that
\[m_{(d-2,d-1),(0,1)}=\frac{(1-d)\,f'}{d^5\,(f^d+f)}\neq0,\]
which completes the proof of the first statement.
 
Now for the proof of part (2). In what follows, we view $\phi$ over $K(\mu_d)$. If $d$ is a prime and $f\not\in {K(\mu_d)}^{d}$, then we first show that $\phi^n$ is irreducible over $K(\mu_d)$ for all $n\geq1$. To see this, we must rule out the presence of $d$-powers in the orbit of zero. Suppose that $\phi^n(0)\in {K(\mu_d)}^d$ for some $n\geq2$. Since, $\phi^n(0)=(((f^d+f)^d+f)^d\dots +f)^d+f$ and $\mathcal{O}_{K(\mu_d)}$ is a UFD, we may write $\phi^n(0)=f\cdot g_n$ for some $g_n\in\mathcal{O}_{K(\mu_d)}$ coprime to $f$. It follows that $f$ and $g_n$ must both be $d$-powers, a contradiction. Hence, $\phi^n(0)\not\in {K(\mu_d)}^d$ for all $n\geq1$, and \cite[Theorem 8]{xdc} implies that all iterates of $\phi$ are irreducible over $K(\mu_d)$. 
 
As for the Galois groups of iterates of $\phi$, note that Lemma \ref{lemma:borrowed} (applied inductively to $\psi=\phi$ and $\gamma=\phi$) implies that $G_{K(\mu_d)}(\phi)\leq W(d)$. On the other hand, we see that the proofs of Theorem \ref{thm:Avg} and Theorem \ref{thm:unicrit} part (1) imply that the $d$-free part of all but finitely many terms of the orbit $\mathcal{O}_{\phi}(b)$ contains primitive prime divisors whenever $\hat{h}_\phi(b)\neq0$. In particular, since $\deg(\phi^n(0))=d^n\cdot\deg(f)$ goes to infinity, this holds for $b=0$. Hence, if $K_m(\phi)$ is a splitting field of $\phi^m$ over $K(\mu_d)$, then $\Gal(K_n(\phi)/K_{n-1}(\phi))\cong(\mathbb{Z}/d\mathbb{Z})^{d^{n}}$ for all but finitely many $n$; see \cite[Theorem 25]{xdc}. Therefore, $G_{K(\mu_d)}(\phi)\leq W(d)$ is a finite index subgroup as claimed.                           
\end{proof}
\begin{remark}{\label{rem:eventuallystable}} It follows from Theorem \ref{thm:unicrit} that $G_{K(\mu_d)}(\phi_f)\leq W(d)$ is a finite index subgroup for all non-constant $f\in K$. To see this, apply Theorem \ref{thm:unicrit} to the field $K_0=\mathbb{F}_p(f)\cong\mathbb{F}_p(t)=K$ and the polynomial $\phi(x)=x^2+t$ and then use the fact that $[\mathbb{F}_p(t):\mathbb{F}_p(f)]= h(f)$ to get the index bound
\[\big[W(d):G_{K(\mu_d)}(\phi_f)\big]\leq \big[W(d):G_{K(\mu_d)}(x^d+t)\big]\cdot h(f). \] 
In particular, the number of irreducible factors of $\phi_f^n$ over $K(\mu_d)$ (and hence over $K$) is bounded independently of $n$ (cf. \cite[Corollary 7]{xdc}).  
\end{remark}
\begin{remark}{\label{eg:isotriv}} We note that $C_{\ell,1}(\phi): Y^\ell=X^d+f$ is isotrivial for all $\ell\geq2$: the map $(X,Y)\rightarrow\Big(\frac{X}{\sqrt[d]{f}}\,,\frac{Y}{\sqrt[\ell]{f}}\Big)$ is an isomorphism (defined over $\widebar{K}$) onto the curve $Y^\ell=X^d+1$. Alternatively, one can compute the Kodaira-Specer map for the surface $C_{\ell,m}(\phi)$. For instance, when $\ell=2$, it follows from \cite[\S5.2]{Voloch} that
\[m_{i,j}=\frac{f'\sqrt[d]{-f}}{2d^2}\cdot\sum_{\zeta\in\mu_d}\zeta^{(y_{ij})}, \,\;\;\text{for some}\; -2d+2\leq y_{ij}\leq -d-1, \,\;y_{ij}\in\mathbb{Z}.\] 
In particular, we see that $y_{ij}\not\equiv{0}\Mod{d}$, from which it follows that $m_{ij}=0$ for all indices $0\leq i,j\leq \frac{d-1}{2}-1$; here again we use the standard basis $\{\frac{x^idx}{y}\}$ of regular differentials on a hyperelliptic curve. In either case, the examples $\phi(x)=x^d+f$ underscore the importance of passing to an iterate (and its corresponding superelliptic curve) to study primitive prime divisors.
\end{remark}
\begin{remark} Although it was enough to show that certain $KS$ maps were non-zero to prove Theorems \ref{thm:Avg} and \ref{thm:unicrit}, we believe that the $KS$ maps for the curves $C_{2,2}(\phi): Y^2=(x^d+f)^d+f$ and $C_{d,2}(\phi): Y^d=(x^d+f)^d+f$ have maximal rank, from which it follows that the best possible height bounds (involving a main term of $2+\epsilon$) hold in these families; see \cite{Kim}. If such a result were true, then one could give relatively small bounds for the size of the elements of $\mathcal{Z}(\phi,S)$. In practice, one exploits the fact the $KS$ matrix is symmetric (\ref{MatrixSimp}) and the fact that the cyclotomic sums $\sum\zeta_s^m$ and power sums $\sum P^{m}$ vanish (a trace-zero fact), to prove that the $KS$ map has maximal rank.   
\end{remark}   
As for characteristic zero function fields, we can make the index bounds explicit and uniform when $K=k(t)$ is a rational function field; compare to uniform bounds in the quadratic case \cite{GaloisUniform}. Here we use work of Schmidt \cite{schmidt} and Mason \cite{Mason} on Thue Equations over function fields; conveniently, we need not worry about isotriviality, since it does not affect the height bounds in this setting.     
\uniform*
\begin{proof} We work over the ground field $K(\mu_d)$. We have already seen that $f\notin K(\mu_d)^d$ implies all iterates of $\phi(x)$ are irreducible over $K(\mu_d)$; see the proof of Theorem \ref{thm:unicrit} above. We show that for all $n\geq11$, there is a place $v_n\in V_{K(\mu_d)}$ such that $v_n(\phi^n(0))>0$,\, $v(\phi^n(0))\not\equiv0\Mod{d}$ and $v(\phi^m(0))=0$ for all $1\leq m\leq n-1$.
 
If there is no such place for some $n\geq2$, then $\phi^n(0)=d_n\cdot y_n^d$ for some ($d$-power free) $d_n$ satisfying
\[ \deg(d_n)\leq (d-1)\sum_{i=1}^{\lfloor\frac{n}{2}\rfloor} \deg(\phi^i(0))=(d-1)\cdot\sum_{i=1}^{\lfloor\frac{n}{2}\rfloor} \deg(f)\cdot d^{i-1}=\deg(f)\cdot(d^{\lfloor\frac{n}{2}\rfloor+1}-1);\]
see (\ref{refinement}) in the proof of Theorem \ref{thm:Avg}.
Hence the curve $C_\phi^{(d_n)}: Y^d=d_n^{d-1}\cdot(X^d+f)$ has an integral point $(\phi^{n-1}(0), d_n\cdot y_n)$. Let $K_1$ be a splitting field of $\phi$ over $K(\mu_d)$, let $\mathfrak{g}_{K_1}$ be the genus of $K_1$, and let $\mathfrak{r}_{K_1}$ be the number of infinite places of $K_1$. Then It follows from \cite[Theorem 15]{Mason} that
\[h_{K_1}(\phi^{n-1}(0))\leq 18h_{K_1}(C_\phi^{(d_n)})+ 6\mathfrak{g}_{K_1}+ 3\mathfrak{r}_{K_1}-3;\]
here $h_{K_1}(C_\phi)$ is the maximum height (relative to $K_1$) of the coefficients defining $C_\phi^{(d_n)}$. However, \cite[2.11]{schmidt} implies that $h_{K_1}(\alpha)=d\cdot\deg(\alpha)$ for all $\alpha\in K(\mu_d)$, and \cite[Lemma H]{schmidt} implies that $\mathfrak{g}_{K_1}\leq(d-1)(\deg(f)-1)$. In particular, we see that
\[d\cdot d^{n-2}\cdot\deg(f)\leq18\cdot d\cdot\deg(f)\cdot d^{\lfloor\frac{n}{2}\rfloor+1}+6(d-1)(\deg(f)-1)+3d-3.\]
Therefore,
\[d^{n-1}\leq 18d^{\lfloor\frac{n}{2}\rfloor+2}+(3d-3)\frac{2\deg(f)-1}{\deg(f)}\leq 18d^{\lfloor\frac{n}{2}\rfloor+2}+6d-6.\]
We deduce that $n\leq 2\log_d(19)+5<10.4$ since $d\geq3, n\geq2$; statements (2) and (3) then follow from \cite[Theorem 2.5]{xdc}.
 
As for the index bound in statement (1), let $K_n(\phi)$ be the splitting field of $\phi^n$ over $K(\mu_d)$. One computes inductively via \cite[\S3.3 Theorem 19]{DF} that $[C_d]^n$ is a group of order $d^{\frac{d^n-1}{d-1}}$ for all $n\geq1$. On the other hand, since the subextensions $K_n(\phi)/K_{n-1}(\phi)$ are Kummer extensions of degree $d^{d^{n-1}}$ for all $n\geq11$ by \cite[Theorem 2.5]{xdc} and the first part of our proof, we have the index bound:
\[\log_d\big[\;[C_d]^n: \Gal_{K(\mu_d)}(\phi^n)\big]=\log_d\frac{d^{\frac{d^n-1}{d-1}}}{[K_{10}(\phi):K(\mu_d)]\cdot \prod_{j=10}^{n-1}[K_{j+1}(\phi):K_j(\phi)]}.\]
However, $[K_{10}(\phi):K(\mu_d)]\geq d^{10}$, since $\phi^{10}$ is an irreducible polynomial, and we deduce that    
\[\log_d\big[\;[C_d]^n: \Gal_{K(\mu_d)}(\phi^n)\big]\leq \frac{d^n-1}{d-1}-(d^{n-1}+d^{n-2}+\dots d^{10}+10)\leq \frac{d^{10}-1}{d-1}+10.\]
 
Finally, we consider the special case $\phi(x)=x^d+t$ for all $d\geq2$ (not necessarily prime). We note that the discriminant of $\phi^n(0)$ (as an element of $k[t]$) is non-zero. In particular, $\phi^n(0)$ is square-free for all $n$. To see this, let $p|d$ be a prime. Then the quotient map $\mathbb{Z}\rightarrow\mathbb{Z}/p\mathbb{Z}$ induces a ring homomorphism $\mathbb{Z}[t]\rightarrow(\mathbb{Z}/p\mathbb{Z})[t]$ given by reducing coefficients. Therefore, if $\phi^n(0)$ is not square-free in $\widebar{\mathbb{Q}}[t]$ (hence not square-free in $\mathbb{Z}[t]$), then the image of $\phi^n(0)\in(\mathbb{Z}/p\mathbb{Z})[t]$ is not square-free. Hence, \cite[\S13.5 Prop. 33]{DF} implies that $\phi^n(0)$ and its formal derivative in $(\mathbb{Z}/p\mathbb{Z})[t]$ must share a root. However. one sees that the formal derivative of $\phi^n(0)$ is $1$ for all $n\geq2$ by the power-rule. We deduce that $\phi^n(0)$ is square-free in $k[t]$ for all $n\ge1$. It terms of valuations, this means that $v_{\mathfrak{p}}(\phi^n(0))=1$ for all $\mathfrak{p}\in V_K$ such that $v_{\mathfrak{p}}(\phi^n(0))>0$. On the other hand, a simple degree computation shows
\[\deg(\phi^n(0))=d^{n-1}>\frac{d^{n-1}-1}{d-1}=d^{n-2}+d^{n-3}+\dots +1=\deg(\phi^{n-1}(0))+\dots +\deg(\phi(0)).\]
Therefore, it is impossible that all prime factors of $\phi^n(0)$ come from lower order iterates. Moreover, since $\phi^n(0)$ is square-free, it follows from \cite[Theorem 8]{xdc} that all iterates of $\phi$ are irreducible over $K(\mu_d)$. We deduce from \cite[Theorem 25]{xdc} that $G_{K(\mu_d)}(\phi)\cong W(d)$ as claimed. The statement about specializations follows from Hilbert's irreducibility theorem \cite[Theorem 1, Theorem 3.4.1]{Serre-Galois}           
\end{proof} 
\begin{remark} The proof of Theorem \ref{thm:uniform} implies that
\[\max\big\{\,n\;\big|\; n\in\mathcal{Z}(x^d+f,0,d)\;\text{for some}\; f\in k[t],\;\deg(f)\geq1, \; d\geq2\,\big\}\leq10 ;\]
here $\mathcal{Z}(x^d+f,0,d)$ is the $d$-th Zsigmondy set \cite[Definition 2]{primdiv}, representing terms which do not have $d$-power free primitive prime divisors. Equivalently, after the $10$th stage of iteration, we always see $d$-power free primitive prime in the orbit of zero (independent of both $f\in k[t]$ and $d$) in characteristic zero (c.f \cite[Theorem 1.1]{Krieger}).  
\end{remark}
\begin{remark} It is tempting to think that the discriminant trick we used to prove $G_{K(\mu_d)}(\phi)\cong W(d)$ for $\phi(x)=x^d+t$ works for all $\phi(x)=x^d+f$ satisfying $f\not\in K(\mu_d)^d$. However, surjectivity already fails for $d=2$: when $\phi(x)=x^2-(t^2+1)$, we show in \cite{GaloisUniform} that $[W(2):G_{K}(\phi)]=2$, even though $-(t^2+1)$ is never a square in characteristic zero; this example is essentially due to Stoll \cite{Stoll-Galois}. Likewise, this discriminant trick does not work in prime characteristic: when $d=5$, $p=43$ and $\phi(x)=x^5+t$, the discriminant of $\phi^6(0)$ is zero in $\mathbb{F}_p[t]$. Nevertheless, a finite index statement (not necessarily surjective) holds by Theorem \ref{thm:unicrit}.  
\end{remark} 
\begin{remark} To the author's knowledge, there is not a single pair $(d,c)$ of values $d\geq3$ and $c\in\mathbb{Q}$ for which $\Gal_{\mathbb{Q}(\mu_d)}(\phi_c^n)$ is known for large $n$. Therefore, Theorem \ref{thm:uniform} represents some progress and solves the inverse Galois problem for $[C_d]^n$. 
\end{remark}
\begin{remark} To prove the surjectivity of the $\ell$-adic Galois representation attached to an elliptic curve, it suffices to prove the surjectivity onto some finite quotient. Namely, if $G\leq\GL_2(\mathbb{Z}_\ell)$ is a closed subgroup that surjects onto $\GL_2(\mathbb{Z}/\ell^n\mathbb{Z})$ for some small $n$, then $G$ must be equal to $\GL_2(\mathbb{Z}_\ell)$; see, for instance, \cite{surj}. In particular, this is a fact from group theory. On the other hand, such a property will fail in general for closed subgroups of $W(d)$. Nevertheless, we have proven this rigidity for subgroups $G_K(\phi)\leq W(d)$ coming from dynamics in Theorem \ref{thm:uniform}.      
\end{remark}  
\end{section}

\vspace{12mm}
\indent\indent Wade Hindes, Department of Mathematics, CUNY Graduate Center, 365 Fifth Avenue, New York, New York 10016-4309.\\
\indent \emph{E-mail address:} \textbf{whindes@gc.cuny.edu}

\begin{thebibliography}{13}
\bibitem{Baker} M. Baker, A finiteness theorem for canonical heights attached to rational maps over function fields, \emph{Journal f\:{u}r die reine und angewandte Mathematik, Crelle} 626 (2009): 205-233.
\bibitem{Bang} A. S. Bang, Taltheoretiske Undersogelse, \emph{Tidsskrift Mat.}, 4.5 (1886): 70-80, 130-137.
\bibitem{BJ} N. Boston and R. Jones, The Image of an arboreal Galois representation, \emph{Pure and Applied Mathematics Quarterly} 5.1 (Special Issue: in honor of Jean-Pierre Serre, Part 2 of 2) (2009): 213-225.
\bibitem{Rafe} J. Cahn, R. Jones, and J. Spear, Powers in orbits of rational functions: cases of an arithmetic dynamical Mordell-Lang conjecture, preprint arXiv:1512.03085
\bibitem{DF} D. Dummit and R. Foote, Abstract algebra, Vol. 1984. Hoboken: Wiley, 2004.
\bibitem{Faltings} G. Faltings, Endlichkeitss{\"a}tze f{\"u}r abelsche Variet{\"a}ten {\"u}ber Zahlk{\"o}rpern, \emph{Inventiones mathematicae}, 73.3 (1983): 349-366.
\bibitem{ABCimplies} C.Gratton, K.Nguyen, and T.Tucker, ABC implies primitive prime divisors in arithmetic dynamics, \emph{Bull. London Math. Soc. 45} (2013): 1194-1208.
\bibitem{xdc} S. Hamblen, R. Jones, and K. Madhu, The density of primes in orbits of $z^d+ c$, \emph{Int. Math. Res. Not.} 7 (2015): 1924-1958.
\bibitem{Hart} R. Hartshorne, Algebraic geometry, Vol. 52. Springer Science and Business Media, 1977.
\bibitem{Me-Avg} W. Hindes, The average number of integral points in orbits, preprint, arXiv:1509.00752 (2015).
\bibitem{primdiv} W. Hindes, Prime divisors in polynomial orbits over function fields, preprint: arXiv:1504.08334.
\bibitem{GaloisUniform} W. Hindes, Galois uniformity in quadratic dynamics over $k(t)$, \emph{J. Number Theory} 148 (2015): 372-383.
\bibitem{Ingram-Stab} P. Ingram, Arboreal Galois representations and uniformization of polynomial dynamics, \emph{Bull. London Math. Soc}. 45.2 (2013):  301-308. 
\bibitem{Ing-Silv} P. Ingram and J. Silverman, Primitive divisors in arithmetic dynamics, \emph{Mathematical Proceedings of the Cambridge Philosophical Society}. 146. No. 02. Cambridge University Press, 2009.
\bibitem{ellipticdivis} P. Ingram and J. Silverman, Uniform estimates for primitive divisors in elliptic divisibility sequences, \emph{Number theory, analysis and geometry}. Springer US, 2012. 243-271.
\bibitem{RJSurvey} R. Jones, Galois representations from pre-image trees: an arboreal survey, \emph{Pub. Math. Besan\c{c}on} (2013): 107-136.
\bibitem{RJones} R. Jones, The density of prime divisors in the arithmetic dynamics of quadratic polynomials, \emph{J. Lond. Math. Soc.} 78.2 (2008): 523-544.
\bibitem{Alon-Rafe} R. Jones and A. Levy, Eventually stable rational functions, preprint, arXiv:1603.00673.
\bibitem{Rafe-Manes} R. Jones and M. Manes, Galois theory of quadratic rational functions, \emph{Comment. Math. Helv}. 89.1 (2014): 173-213.
\bibitem{Juul} J. Jamie, P. Kurlberg, K. Madhu, and T. Tucker, Wreath products and proportions of periodic points, \emph{Int. Math. Res. Not.} (2015): rnv273.
\bibitem{Kim} M. Kim, Geometric height inequalities and the Kodaira-Spencer map, \emph{Compositio Math}. 105.1: 43-54 (1997).
\bibitem{Voloch} M. Kim, D. Thakur, and J. Voloch, Diophantine approximation and deformation, \emph{Bulletin de la Soci\'{e}t\'{e} Math\'{e}matique de France} 128.4 (2000): 585-598.
\bibitem{differentials} J. Koo, On holomorphic differentials of some algebraic function field of one variable over $\mathbb{C}$, \emph{Bulletin of the Australian Mathematical Society} 43.03 (1991): 399-405.
\bibitem{Krieger} H. Krieger, Primitive Prime Divisors in the Critical Orbit of $z^d+ c$, \emph{Int. Math. Res. Not}. 23 (2012): 5498-5525.
\bibitem{Mason} R.C.Mason, Diophantine equations over function fields, Cambridge University Press,1984.
\bibitem{Cheba} V. Kumar Murty and J. Scherk, Effective versions of the Chebotarev density theorem for function fields, \emph{C. R. Acad. Sci. Paris S\'{e}r. I Math}. 319.6 (1994): 523-528.
\bibitem{wreath} A. Mikl\'{o}s and B. Vir\'{a}g, Dimension and randomness in groups acting on rooted trees, \emph{J. Amer. Math. Soc}. 18.1 (2005): 157-192.
\bibitem{htineq} A. Moriwaki, Geometric height inequality on varieties with ample cotangent bundles, \emph{J. Algebraic Geom.}, 4.2 (1995): 385-396.
\bibitem{Odoni} R. W. K. Odoni, The Galois theory of Iterates and Composites of Polynomials, \emph{Proc. London Math. Soc.}, 51.3 (1985): 385-414.
\bibitem{Pink} R. Pink, Profinite iterated monodromy groups arising from quadratic polynomials, preprint arXiv:1307.5678. 
\bibitem{Rice} B Rice, Primitive prime divisors in polynomial arithmetic dynamics, \emph{Integers} 7.1 (2007): 26.
\bibitem{Rosen} M. Rosen, Number theory in function fields. Vol. 210. Springer-Verlag, (2002).
\bibitem{Schinzel} A. Schinzel, Primitive divisors of the expression $a^n-b^n$ in algebraic number fields, \emph{J. Reine Angew}. 268/269 (1974), 27-33, Collection of articles dedicated to Helmut Hasse on his seventy-fifth birthday, II.
\bibitem{schmidt} W. Schmidt, Thue's equation over function fields, \emph{Journal of the Australian Mathematical Society}, Society. (Series A) 25.04 (1978): 385-422.
\bibitem{Serre-reps} J.-P. Serre, Abelian $\ell$-adic representations and elliptic curves, W.A. Benjamin
Inc. New York-Amsterdam 1968.
\bibitem{Serre} J-P. Serre, Lectures on $N_X (p)$, CRC Press, 2011.
\bibitem{Serre-Galois} Serre, Jean-Pierre. Topics in Galois theory, CRC Press, 2007.
\bibitem{SilvA} J. Silverman, Advanced Topics in the Arithmetic of Elliptic Curves, Graduate Texts in Math. 151, Springer (1994).
\bibitem{Silv-Ellip} J. Silverman, The Arithmetic of Elliptic Curves, Graduate Texts in Mathematics, Springer-Verlag, GTM 106, 1986, Expanded 2nd Edition, 2009.
\bibitem{Silv-Duke} J. Silverman, Integer points, Diophantine approximation, and iteration of rational maps, \emph{Duke Math}. J 71.3 (1993): 793-829.
\bibitem{Silv-Vojta} J. Silverman, Primitive divisors, dynamical Zsigmondy sets, and Vojta's conjecture, \emph{Journal of Number Theory} 133.9 (2013): 2948-2963.
\bibitem{Silv-Dyn} J. Silverman, The Arithmetic of Dynamical Systems, Vol. 241, Springer, 2007.
\bibitem{Sookdeo} V. Sookdeo, Integer points in backward orbits, \emph{J. Number Theory}, 131.7 (2011): 1229-1239.  
\bibitem{ffields} H. Stichtenoth, Algebraic function fields and codes, Springer-Verlag, Vol. 254, 1993.
\bibitem{Stoll-Galois} M. Stoll, Galois groups over $\mathbb{Q}$ of some iterated polynomials, \emph{Arch. Math.} 59 (1992): 239-244.
\bibitem{surj} A. Vasiu, Surjectivity criteria for $p$-adic representations, Part 1 \emph{Manuscripta Mathematica} 112.3 : 325-355. 
\bibitem{Vojta} P. Vojta, Diophantine approximation and Nevanlinna theory, Arithmetic geometry,  \emph{Springer Berlin Heidelberg}, 2010: 111-224.
\bibitem{Voloch-Survey} J.-F. Voloch, \emph{Diophantine geometry in characteristic $p$: a survey}, Arithmetic geometry (Cortona 1994), Sympos. Math. XXXVII (Cambridge University Press, 1997) 260-278.
\bibitem{Zsig} K. Zsigmondy, Zur Theorie der Potenzreste, \emph{Monatsch. Math. Phys}. 3 (1892): 265-284.
\end{thebibliography}
\end{document}